\documentclass[12pt]{amsart} \usepackage{amsmath} \usepackage{enumerate} \usepackage{amssymb} \usepackage{amsbsy} \usepackage{amsfonts} \textwidth
5.5truein

\theoremstyle{plain}  \newtheorem{prop}{Proposition}[section] \newtheorem{lem}[prop]{Lemma} \newtheorem{thr}[prop]{Theorem}
 \theoremstyle{definition}   \newtheorem{rem}[prop]{Remark}
  \newtheorem*{re*}{Remark} \newtheorem*{ex*}{Example} \numberwithin{equation}{section}

\newcommand{\sm}{\begin{smallmatrix}} \newcommand{\esm}{\end{smallmatrix}} \newcommand{\Z}{\mathbb{Z}} \newcommand{\Q}{\mathbb{Q}}
    
 \newfont{\FieldFont}{msbm10 scaled\magstep1}  
  
  \newcommand{\trans}[1]{\,^t\hspace{-0.1em}{#1}}

\begin{document}

\title{Bounds for Siegel Modular Forms of genus 2 modulo $p$}

\author{D. Choi} \address{School of Liberal Arts and Sciences, Korea Aerospace University, 200-1, Hwajeon-dong, Goyang, Gyeonggi 412-791, Korea}
\email{choija@kau.ac.kr}

\author{Y. Choie
}
\address{Department of Mathematics \\ Pohang University of Science and Technology\\Pohang\\790-784\\Korea}
\email{yjc@postech.ac.kr}

\author{T. Kikuta} \address{Department of Mathematics \\ Interdisciplinary Graduate School of Science and Engineering \\Kinki University\\Higashi-Osaka\\577-8502\\Japan} \email{kikuta84@gmail.com}


\thanks{The second author was  partially supported by   NRF
2009008-3919 and NRF-2010-0029683}

\begin{abstract} Sturm \cite{S} obtained the bounds for the number of the first Fourier coefficients of elliptic modular form $f$ to determine
vanishing of $f$ modulo a prime $p$. In this paper, we study analogues of Sturm's bound for Siegel modular forms of genus 2. We show the resulting
bound is sharp. As an application, we study congruences involving Atkin's $U(p)$-operator for the Fourier coefficients of Siegel mdoular forms of
genus 2. \end{abstract}

 \maketitle

\section{\bf Introduction and Statement of Main Results}

\medskip

Despite of a vast research result related to congruence properties of  Fourier coefficients
of elliptic modular forms it is surprising to know a very little result are available about congruence properties of Siegel modular forms of higher
genus.  

In this paper, we consider congruence properties for the Fourier coefficients of Siegel modular forms of genus 2.
This is an application of
Sturm's formula:
 let $f=\sum_{n=0}^{\infty}a_f(n)q^n \in \mathcal{O}_L[[q]]\cap
M_k(\Gamma)$ be an elliptic modular form of weight $k$ on a congruence subgroup $\Gamma$, where $\mathcal{O}_L$ denotes the ring of integers of a
number field $L$. In \cite{S} Sturm proved that if $\beta$ is a prime ideal of $\mathcal{O}_L$ for which
 $a_f(n) \equiv 0 \pmod{\beta}$ for $0 \leq n \leq \frac{k}{12}[ \Gamma(1)
: \Gamma ],$ then $a_f(n) \equiv 0 \; (mod \; \beta)$ for every $n\geq 0.$ This formula is called as Sturm's formula.  In fact, Sturm's formula
implies that $f$ is determined modulo $\beta$ by the first $\frac{k}{12}[ \Gamma(1):\Gamma ]$ coefficients. Using the ring structure of weak Jacobi
forms, the first two authors in \cite{C-C} derived the Sturm's formula in the case of Jacobi forms.

Our first main Theorem \ref{main1} states an analogous result of Sturm's formula in the space of
 Siegel modular forms of genus 2.
Note that Poor and Yuen in \cite{P-Y} gave a result in
this direction which is more general, but different (see Remark \ref{PY}-(2)). Our result was obtained independently. \\

Let $\Gamma_2$ be the full Siegel modular group $\Gamma_2=Sp_2(\mathbb{Z})$ and $\Gamma$ be a congruence subgroup of $\Gamma _2$ with level $N$. Let
$M_2(\Gamma)$ be the space of Siegel modular forms of weight $k$ on $\Gamma$ and let $S_k(\Gamma)$ denote the space of cusp forms in $M_k(\Gamma)$. In
the following theorem, we obtain analogues of Sturm's formula for Siegel modular form of genus 2.

\begin{thr}\label{main1} Let $k$ be an even positive integer. Suppose that $p\geq5$ is a prime and that $F\in M_k(\Gamma)$ with $p$-integral rational
coefficients such that $$ F(\tau,z,\tau')=\sum_{\begin{smallmatrix} n,r,m \in \frac{1}{N}\mathbb{Z}\\ n,m,4nm-r^2\geq 0 \end{smallmatrix}}A(n,r,m)q^n
\xi^r q'^m,$$ with $ q=e^{2\pi i \tau}, \xi=e^{2\pi i z}, q'= e^{2\pi i \tau'}.$ If $A(n,r,m)\equiv 0 \pmod{p}$ for every $n$, $m$ such that $$0\leq n \leq \frac{k}{10}[\Gamma_2:\Gamma] \text{ and }0\leq m \leq
\frac{k}{10}[\Gamma_2:\Gamma],$$ then $F\equiv0 \pmod{p}$. \end{thr}

\begin{rem}\label{PY}  \begin{enumerate} \item The above bounds in Theorem \ref{main1} are sharp for the case that the level is one, and this will be discussed in section \ref{example}.

\item The results of Poor and Yuen in \cite{P-Y} are stated in term of the dyadic trace. For example, when the level is one, their results implies the following:

\noindent {\bf{Theorem}}.  \cite{P-Y} Let $F=\sum A(T)e^{\pi i tr(TZ)}\in M_k^{(2)}$ be such
    that for all $T$ with dyadic trace $w(T)\leq \frac{k}{3}$ one has that $A(T) \in {\mathbb{Z}}_{(p)}$ and $A(T)\equiv 0\pmod{p}.$ Then $F\equiv
    0\pmod{p}.$

\end{enumerate} \end{rem} \medskip

We apply Theorem \ref{main1} to study congruences involving an analogue of Atkin's $U(p)$-operator. For elliptic modular forms congruences, involving
Atkin's $U(p)$-operator were studied with important applications in the context of traces of singular moduli and class equations (see Ahlgren and Ono
\cite{A-O}, Elkies, Ono, and Yang \cite{E-O-Y}, and chapter $7$ of Ono \cite{O}). Recently analogues of the results for Siegel modular forms were
studied in \cite{C-C-R}. Using Theorem \ref{main1}, we improve the result in \cite{C-C-R}.

Before stating main theorem we set up the notation: $$ \widetilde{M}_k:=\left\{F \hspace{1ex}(\mbox{mod } p) \,:\, F(Z)= \sum a(T)e^{\pi i\,tr(TZ)}\in
M_k(\Gamma_2) \,\mbox{ where } a(T)\in\Z_{(p)}\right\}, $$ where $\Z_{(p)}:=\Z_p\cap\Q$ denotes the local ring of $p$-integral rational numbers. For
Siegel modular forms with $p$-integral rational coefficients, we define their filtrations modulo $p$ by \begin{equation*} 
\omega\big(F\big):=\inf\left\{k\,:\, F \hspace{1ex}(\mbox{mod } p)\,\in \widetilde{M}_k\right\}. \end{equation*} The differential operator
$\mathbb{D}=(2\pi)^{-2}(4\frac{\partial }{\partial \tau}\frac{\partial}{\partial \tau'}-\frac{\partial^2}{\partial z^2}), Z= \left(\begin{array}{cc} \tau & z\\ z&
\tau'\end{array}\right),$ which acts on Fourier expansions of Siegel modular forms as $$\mathbb{D}(\sum a(T)e^{\pi i\,tr(TZ)}) =\sum  \det(T) a(T)e^{\pi
i\,tr(TZ)}.$$

\begin{thr}\label{main2} Let \begin{equation*} F(Z)= \sum_{\begin{smallmatrix}T=\trans{T}\geq 0 \\ T\, even\end{smallmatrix}}a(T)e^{\pi i\,tr(TZ)}=
\sum_{\begin{smallmatrix} n,r,m \in \mathbb{Z}\\ n,m,4nm-r^2\geq 0 \end{smallmatrix}}A(n,r,m)q^n \xi^r q'^m \end{equation*} be a Siegel modular form
of degree $2$, even weight $k$, level 1, and with $p$-integral rational coefficients, where $p>k$ is a prime.  Let \begin{equation*} F(Z)\, \big | \, U(p)
:=\sum_{\begin{smallmatrix}T=\trans{T}\geq 0 \\ T\, even\\p \, \,\mid \, \det T\end{smallmatrix}}a(T)e^{\pi i\,tr(TZ)} \end{equation*} be the analog
of Atkin's $U$-operator for Siegel modular forms. If $p>2k-5$ and $F \not \equiv 0 \pmod{p}$, then $F \, \big | \, U(p) \not\equiv 0 \hspace{1ex}
(\mbox{mod } p)$. If $k<p<2k-5$, then \begin{equation*} \omega\left(\mathbb{D}^{\frac{3p+3}{2}-k}(F)\right)=\left\{\begin{array}{cc} 3p-k+3 &
\mbox{if} \hspace{2ex} F \, \big | \, U(p) \not\equiv 0 \hspace{1ex} (\mbox{mod } p).\\ 2p-k+4 & \mbox{if} \hspace{2ex}  F \, \big | \, U(p) \equiv 0
\hspace{1ex} (\mbox{mod } p),\end{array}.\right. \end{equation*} \end{thr}

\section{\bf{Proofs of main theorems}}\label{proofs} Let $E^{(2)}_4, E^{(2)}_6, \chi_{10}$, and $\chi_{12}$ denote the usual generators (see \cite{Ig}) of
$M_k(\Gamma_2)$ of weights $4$, $6$, $10$, and $12$, respectively, where $E^{(2)}_k(Z)$ is the normalized Eisenstein series of weight $k$ and genus
$2$ ($E^{(2)}_4$ and $E^{(2)}_6$ are normalized by $a\left(\begin{smallmatrix} 0 & 0 \\ 0 & 0\end{smallmatrix}\right)=1$) and where the cusp forms
$\chi_{10}$ and $\chi_{12}$ are normalized by $a\left(\begin{smallmatrix} 2 & 1 \\ 1 & 2\end{smallmatrix}\right)=1$.  Let $p$ be a prime. Nagaoka
\cite{N}, and B\"ocherer and Nagaoka \cite{B-N} investigate Siegel modular forms modulo $p$.

Consider the Witt operator $W:M_k(\Gamma_2)\rightarrow M_k(\Gamma_1)\otimes M_k(\Gamma_1),$ defined by $$W(F)(\tau,\tau'):=F(\left(\sm \tau & 0 \\ 0
& \tau'\esm\right)), \; (\tau,\tau')\in\mathbb{H}\times\mathbb{H}.$$
 For
example, note that(see \cite{N}) \begin{align*} &W(E^{(2)}_4)(\tau,\tau')=E^{(1)}_4(\tau)E^{(1)}_4(\tau')\\
&W(E^{(2)}_6)(\tau,\tau')=E^{(1)}_6(\tau)E^{(1)}_6(\tau')\\ &W(\chi_{10})(\tau,\tau')=0, \;\; W(\chi_{12})=\Delta(\tau)\Delta(\tau'). \end{align*}

\begin{lem}[Corollary 4.2 in \cite{N}]\label{Na-lem} Let p be a prime number. If $F\in M_k(\Gamma_2)_{\mathbb{Z}_{(p)}}$ satisfies $W(F)\equiv 0
\pmod{p}$, then $\frac{F}{\chi_{10}} \in M_{k-10}(\Gamma_2)_{\mathbb{Z}_{(p)}}$. \end{lem}

We recall the Sturm's formula for Jacobi forms \cite{C-C}.

\begin{lem}[Theorem 1.3 in \cite{C-C}]\label{Jacobi} Suppose that $k$ is even and
 $$\Psi_{k,m}(\tau,z)=\sum_{r^2 \leq 4mn}a(n,r)q^n\zeta^r \in
\mathcal{O}_L[\zeta,\zeta^{-1}][[q]] \cap J_{k,m}(\Gamma).$$ Here $J_{k,m}(\Gamma) $ denotes the space of Jacobi forms of weight $k$ and index $m.$
 Let $\beta$ denote a prime ideal of $\mathcal{O}_L$ over a prime $p
\geq 5$. If $$ord_1[\beta]\left( \sum_{r^2 \leq 4mn}a(n,r)\zeta^r \right) \geq 2m+1$$ for  $0 \leq n \leq \frac{1}{12}(k+2m)[\Gamma(1): \Gamma],$
then
 $\Psi_{k,m}(\tau,z) \equiv 0 \pmod{\beta}.$
\end{lem}

With the above lemmata, we prove Theorem \ref{main1} for the case that level is 1.

\begin{prop}\label{level1} Suppose that $\Gamma=\Gamma_2$. Then Theorem \ref{main1} is true. \end{prop} \begin{proof} Suppose that $A(n,r,m)\equiv 0
\pmod{p}$ for every $n$, $m$ such that $0\leq n \leq \frac{k}{10}$ and $0\leq m \leq \frac{1}{10}k$. Then, $A(n,r,m)\equiv 0 \pmod{p}$ for every $n$,
$m$ such that $0\leq n \leq \frac{1}{12}(k+2m)$ and $0\leq m \leq \frac{1}{10}k$. Let $\displaystyle F(\tau,z,\tau')= \sum_{m=0}^{\infty}
\phi_m(\tau,z) e^{2\pi i m\tau'}$ be the Fourier-Jacobi expansion of $F$ such that $\phi_m$ is a Jacobi form of weight $k$ and index $m$. For every
$m$ such that $0\leq m \leq \frac{1}{10}k$, Lemma \ref{Jacobi} implies that $\phi_m\equiv 0 \pmod{p}.$ Thus, we have that \begin{equation}\label{1.1}
A(n,r,m)\equiv0 \pmod{p} \; \; \text{ if } 0\leq m \leq \frac{1}{10}k. \end{equation} Let $W(F)$ be the Witt operator defined as
$$W(F(Z))=F(\tau,0,\tau').$$ Then we have $$W(F(Z))=\sum_{\alpha=1}^{\beta}f_{\alpha}(\tau)f_{\alpha}(\tau'),$$ where $f_{\alpha}$ is a modular form
of weight $k$ on $SL_2(\mathbb{Z})$. From the ring structure of modular forms on $SL_2(\mathbb{Z})$, we have \begin{align*} W(F(Z))=\sum_{ \sm
12i+4j+6t=k\\ t=0,1 \esm }f_{(i)}(\tau)\Delta(\tau')^i E_4(\tau')^j E_6(\tau')^t. \end{align*} Note that the $q$-expansion of $\Delta(\tau')^i
E_4(\tau')^j E_6(\tau')^t$ has the form $$\Delta(\tau')^i E_4(\tau')^j E_6(\tau')^t=(q')^i+\cdots$$ ans that $j$ and $t$ are uniquely determined by
choosing a value of $i$. Thus, Sturm's formula (\ref{1.1}) implies that if $i\leq \frac{k}{12}\leq \frac{k}{10}$, then $f_{(i)}(\tau)\equiv 0
\pmod{p}.$ Since $12i+4j+6t=k$, we have $W(F)\equiv 0 \pmod{p}$. Let $\chi_{10}$ be the unique (normalized by $a\left(\begin{smallmatrix} 2 & 1 \\ 1 &
2\end{smallmatrix}\right)=1$) Siegel cusp form of degree $2$ and weight $10$. From Lemma \ref{Na-lem} we have $$F\equiv F_{(1)}(Z)\chi_{10}(Z)
\pmod{p},$$ where $F_{(1)}(Z)$ is a Siegel modular form of weight $k-10$ and genus 2.

Let $\displaystyle F_{(1)}(\tau,z,\tau')= \sum_{m=0}^{\infty} \phi_m^{(1)}(\tau,z) e^{2\pi i m\tau'}$ be the Fourier-Jacobi expansion of $F_{(1)}$.
Note that $\chi_{10}(Z)$ has the form of the Fourier-Jacobi expansion $\sum_{m=1}^{\infty} \psi_m(\tau,z) e^{2\pi i m\tau'}$ such that $\psi_1(\tau,z)
\not \equiv 0 \pmod{p}$. Thus, we have that for every $m$, $0\leq m \leq \frac{1}{10}k-1$, $$\phi_m\equiv 0 \pmod{p}.$$ Note that $\frac{1}{10}k-1\geq
\frac{k-10}{12}$. Thus, following the previous argument, we have $$F_{(1)}\equiv F_{(2)}(Z)\chi_{10}(Z) \pmod{p},$$ where $F_{(2)}(Z)$ is a Siegel
modular form of weight $k-20$ and genus 2.

Repeating this argument, we have $$F_{(t)}\equiv  F_{(t+1)}(Z)\chi_{10}(Z)^t \pmod{p}$$ if $1 \leq t < t_0$ and $t_0=\left[\frac{k}{10}\right]+1$. Let
$\displaystyle F_{(t_0)}(\tau,z,\tau')= \sum_{m=0}^{\infty} \phi_m^{(t)}(\tau,z) e^{2\pi i m\tau'}$ be the Fourier-Jacobi expansion of $F_{(t_0)}$.
Since $\frac{1}{10}k-t>0$, we have $\phi_0^{(t)}\equiv 0 \pmod{p}$. Moreover, the weight of $F_{(t_0)}$  is less than 8, and
$\dim_{\mathbb{C}}M_k(\Gamma_2)\leq 1$ for $k\leq10$. Thus, we have $F_{(t_0)}\equiv 0 \pmod{p}$. Since $$F\equiv \chi_{10}^{t_0-1}F_{(t^0)}
\pmod{p},$$ this completes the proof. \end{proof}

For a subring $R\subset \mathbb{C}$, we denote by $M_k(\Gamma )_{R}$ the space of all of $f\in M_k(\Gamma )$ whose Fourier coefficients are in $R$. Let $p$ be a prime. We denote by $v_p$ the normalized additive valuation on $\mathbb{Q}$ (i.e. $v_p(p)=1$). For a Siegel modular form $F(\tau,z,\tau')=\sum_{\begin{smallmatrix}
n,r,m \in \frac{1}{N}\mathbb{Z}\\ n,m,4nm-r^2\geq 0 \end{smallmatrix}}A(n,r,m)q^n \xi^r q'^m \in M_k(\Gamma )_{\mathbb{Q}}$, we define
\begin{align*}
v_p(F):=\inf \left\{v_p(A(n,r,m))\:|\:n,r,m \in \frac{1}{N}\mathbb{Z},\ n,m,4nm-r^2\geq 0 \right \}.
\end{align*}

Assume that $v_p(F)\ge 0$, in other words $F\in M_k(\Gamma )_{\mathbb{Z}_{(p)}}$. Then we define an order of $F$ by
\begin{align*}
{\rm ord}_p(F)&:=\min\{m\:|\:\phi _{m}(\tau ,z)=\sum _{n,r}A(n,r,m)\xi ^rq^n\not\equiv 0 \pmod{p}\}\\ &=\min\{m\:|\: v_p(A(n,r,m))=0\ {\rm for}\ {\rm some}\ n,r \}.
\end{align*}
If $v_p(F)\ge 1$, then we define ${\rm ord}_p(F):=\infty $. Note that
\begin{align} \label{ord2} {\rm ord}_{p}(FG)={\rm ord}_{p}(F)+{\rm ord}_{p}(G). \end{align}
Following the similar argument of Sturm in \cite{S}, we complete the remained part of proof.

\begin{proof}[Proof of Theorem \ref{main1}] Let $i:=[\Gamma_2:\Gamma]$. Suppose that $A(n,r,m)\equiv 0 \pmod{p}$ for every $n$, $m$ such that $0\leq n \leq \frac{1}{10}ki$ and
$0\leq m \leq \frac{1}{10}ki$. Then, $A(n,r,m)\equiv 0 \pmod{p}$ for every $n$, $m$ such that $0\leq n \leq \frac{1}{12}(ki+2m)$ and $0\leq m \leq
\frac{1}{10}ki$. Lemma \ref{Jacobi} implies that ${\rm ord}_p(F)>\frac{1}{10}ki$.

We decompose $\Gamma _2$ as $\Gamma
_2=\cup _{j=1}^{i}\Gamma \gamma _j$, where $\gamma _1=1_2$. Let
\begin{align*}
\Psi :=F\prod _{j=2}^{i}F|_k\gamma _j\in M_{ki}(\Gamma _2)_{\mathbb{Q}(\mu _N)}\quad (\mu _N:=e^{\frac{2\pi i}{N}}).
\end{align*}
The fact that all Fourier coefficients of $\Psi $ are in $\mathbb{Q}(\mu _N)$ is according to Shimura's result \cite{Shi}. We take a constant $\lambda \in \mathbb{Q}(\mu _N)$ such that at least one of the non-zero Fourier coefficients of $\lambda \prod _{j=2}^{i}F|_k\gamma _j$ is in $\mathbb{Q}$. For example, we may take as $\lambda :=A_{\Psi }(n,r,m)^{-1}$ for a Fourier coefficient $A_{\Psi }(n,r,m)\neq 0$ of $\Psi$.

Moreover we consider
\begin{align*}
\Phi &:=\sum _{\sigma \in Gal(\mathbb{Q}(\mu _N)/\mathbb{Q})}(\lambda \Psi )^{\sigma }=\sum _{\sigma \in Gal(\mathbb{Q}(\mu _N)/\mathbb{Q})}(\lambda F \prod _{j=2}^{i}F|_k\gamma _j)^{\sigma }\\
&=F\sum _{\sigma \in Gal(\mathbb{Q}(\mu _N)/\mathbb{Q})}(\lambda \prod _{j=2}^{i}F|_k\gamma _j)^{\sigma },
\end{align*}
where $\sigma $ runs over all elements of $Gal(\mathbb{Q}(\mu _N)/\mathbb{Q})$ and $f ^{\sigma }$ is defined by taking action of $\sigma $ on each Fourier coefficient of $f$ when $f\in M_k(\Gamma )_{\mathbb{Q}(\mu _N)}$. Then we have $\Phi \in M_{ki}(\Gamma _2)_{\mathbb{Q}}$. This modularity follows from Sturm's result \cite{St2}, p.344. It follows from the choice of $\lambda $ that $\Phi \neq 0$. Hence we can take a suitable constant $C\in \mathbb{Q}$ such that
\begin{align*}
v_p\left(C\sum _{\sigma \in Gal(\mathbb{Q}(\mu _N)/\mathbb{Q})}(\lambda \prod _{j=2}^{i}F|_k\gamma _j)^{\sigma }\right)=0.
\end{align*}
This means $C\Phi \in M_{ki}(\Gamma _2)_{\mathbb{Z}_{(p)}}$. Using (\ref{ord2}), we obtain
\begin{align*}
{\rm ord}_p(C\Phi)={\rm ord}_{p}(F)+{\rm ord}_{p}\left(C\sum _{\sigma \in Gal(\mathbb{Q}(\mu _N)/\mathbb{Q})}(\lambda \prod _{j=2}^{i}F|_k\gamma _j)^{\sigma }\right).
\end{align*}
Note here that ${\rm ord}_{p}(C\Phi )\ge {\rm ord}_p(F)>\frac{1}{10}ki$. Thus, we have by Proposition \ref{level1} $${\rm ord}_{p}(C\Phi
)=\infty.$$ This implies
\begin{align*}
{\rm ord}_{p}(F)+{\rm ord}_p\left(C\sum _{\sigma \in Gal(\mathbb{Q}(\mu _N)/\mathbb{Q})}(\lambda \prod _{j=2}^{i}F|_k\gamma _j)^{\sigma }\right)=\infty .
\end{align*}
The second part is finite and hence we have ${\rm ord}_p(F)=\infty$, namely $F\equiv 0$ mod $p$. This completes the proof of Theorem \ref{main1}.
\end{proof}

\begin{proof}[{\bf{Proof of Theorem \ref{main2}}}]
By Theorem 1 in \cite{C-C-R}, it is enough to show that if
$A(n,r,m)\equiv 0 \pmod{p}$ for every $(n,r,m)$ such that $p\nmid
nm$, then $F\equiv 0 \pmod{p}$. To prove it   consider the
Fourier-Jacobi expansion of
$$F(\tau,z,\tau')= \sum_{m=0}^{\infty} \phi_m(\tau,z) e^{2\pi i m\tau'}.$$ Suppose that $A(n,r,m)\equiv 0 \pmod{p}$ for every $(n,r,m)$ such that $p\nmid nm$. If $m<p$ and $n<p$, then $A(n,r,m)\equiv 0 \pmod{p}$. Note that
$\frac{1}{10}k<p$. By Theorem \ref{main1} we complete the proof.
\end{proof}

\section{{\bf{Examples}}} \subsection{Examples for level 1}\label{example} Let $t(k)=\left[\frac{k}{10}\right]$ and $p\geq5$ be a prime. Consider a
Siegel modular

$$G_k(Z)=\sum_{\begin{smallmatrix} n,r,m \in \mathbb{Z}\\ n,m,4nm-r^2\geq 0 \end{smallmatrix}}A_G(n,r,m)q^n \xi^r q'^m$$ of weight $k$ and genus 2
defined as $$ G_k(Z) :=\left\{\begin{array}{cc} E_4^{(2)}(Z)^iE_6^{(2)}(Z)^j\chi_{10}^{t(k)}(Z) \;\; (i+j+t(k)=k,\; j=0 \text{ or } 1) \;   \mbox{\,
if $ k \not \equiv 2 \pmod{10}$},\\
  \chi_{10}^{t(k)-1}(Z)\chi_{12}(Z) \mbox{\, \, if $ k\equiv 2
  \pmod{10}$}
 \end{array}\right\}
$$ Recall that \begin{align*} &E_4^{(2)}(Z) = 1 + 240q + 240q'+\cdots,\\ &E_6^{(2)}(Z)= 1-504q -504q'+\cdots,\\ &\chi_{10}(Z)= (\xi^{-1}-2 +
\xi)qq'+\cdots,\\ &\chi_{12}(Z)= (\xi^{-1} + 10 + \xi)qq'+\cdots. \end{align*} Thus, we have $$ G_k(Z) =\left\{\begin{array}{cc} (\xi^{-1}-2 +
\xi)^{t(k)}q^{t(k)}q'^{t(k)}+\cdots \mbox{\,  if $ k \not \equiv 2 \pmod{10}$},\\
  (\xi^{-1} + 10 + \xi)(\xi^{-1}-2 + \xi)^{t(k)-1}q^{t(k)}q'^{t(k)}+\cdots \mbox{\, if $ k\equiv 2 \pmod{10}$}.
 \end{array}\right\}
$$ The coefficients $A_G(n,r,m)$ are integral and $A_G(n,r,m)\equiv 0 \pmod{p}$ for $n\leq \frac{k}{10}-1$ and $m\leq \frac{k}{10}-1$. This implies
that the bounds in Theorem \ref{main1} are sharp since $G_k \not \equiv 0 \pmod{p}$.

\subsection{Examples for level 11 and 19}\label{example2}

It is known that we can construct a cusp form of $\Gamma _{0}^{(2)}(11)$ of weight $2$ by Yoshida lift (cf. \cite{Yo}). For matrices \begin{align*}
S_1^{(11)}:=\begin{pmatrix}1 & \frac{1}{2} & 0 & 0 \\ \frac{1}{2} & 3 & 0 & 0 \\ 0 & 0 & 1 & \frac{1}{2} \\ 0 & 0 & \frac{1}{2} & 3
\end{pmatrix},\quad S_2^{(11)}:=\begin{pmatrix} 2 & 0 & 1 & \frac{1}{2} \\ 0 & 2 & \frac{1}{2} & -1 \\ 1 & \frac{1}{2} & 2 & 0 \\ \frac{1}{2} & -1 & 0
& 2 \end{pmatrix},\quad S_3^{(11)}:=\begin{pmatrix} 1 & 0 & \frac{1}{2} & 0 \\ 0 & 4 & 2 & \frac{3}{2} \\ \frac{1}{2} & 2 & 4 & \frac{7}{2} \\ 0 &
\frac{3}{2} & \frac{7}{2} & 4 \end{pmatrix}, \end{align*} if we put \begin{align*} F^{(11)}_2:=\frac{1}{24}(3\theta _{S_1^{(11)}}-\theta
_{S_2^{(11)}}-2\theta _{S_3^{(11)}}), \end{align*} then $F^{(11)}_2\in S_2(\Gamma _{0}^{(2)}(11))_{\mathbb{Z}_{(p)}}$, where
$\theta _{S_j}$ is defined by \begin{align*} \theta _{S_j}(Z):=\sum_{X\in M_{4,2}(\mathbb{Z})}e^{2\pi i {\rm tr}(S_j[X]Z)}. \end{align*}

We give one more example of Yoshida lift. For matrices \begin{align*} S_1^{(19)}:=\begin{pmatrix}1 & 0 & \frac{1}{2} & 0 \\ 0 & 1 & 0 & \frac{1}{2} \\
\frac{1}{2} & 0 & 5 & 0 \\ 0 & \frac{1}{2} & 0 & 5 \end{pmatrix},\quad S_2^{(19)}:=\begin{pmatrix} 1 & \frac{1}{2} & \frac{1}{2} & \frac{1}{2} \\
\frac{1}{2} & 2 & 0 & 1 \\ \frac{1}{2} & 0 & 3 & \frac{3}{2} \\ \frac{1}{2} & 1 & \frac{3}{2} & 6 \end{pmatrix},\quad S_3^{(19)}:=\begin{pmatrix} 2 &
0 & 1 & \frac{1}{2} \\ 0 & 2 & \frac{1}{2} & 1 \\ 1 & \frac{1}{2} & 3 & \frac{1}{2} \\ \frac{1}{2} & 1 & \frac{1}{2} & 3 \end{pmatrix}, \end{align*}
if we put \begin{align*} F^{(19)}_2:=\frac{1}{8}(\theta _{S_1^{(19)}}-2\theta _{S_2^{(19)}}+\theta _{S_3^{(19)}}), \end{align*} then $F^{(19)}_2\in
S_2(\Gamma _{0}^{(2)}(19))_{\mathbb{Z}_{(p)}}$.

Let $E_4=E_4^{(2)}$, $E_6=E_6^{(2)}$, $\chi _{10}$ and $\chi _{12}$ be Igusa's generators as introduced in Section \ref{proofs}. We set $\chi _{20}:=11E_4E_6\chi_{10}+4\chi _{10}^2+8E_4^2\chi _{12}$. For a Siegel modular form $F$, we denote by $a_F(T)$ the $T$-th Fourier coefficient of $F$. Then, the following proposition is known.

\begin{prop}[Nagaoka-Nakamura \cite{Na-Na}] \label{Na-Na} The following holds. \\ (1) $a_{F^{(11)}_2}(T)\equiv a_{-\chi _{12}}(T)$ mod $11$ for all $T$
with ${\rm tr}(T)\le 5$.\\ (2) $a_{F^{(19)}_2}(T)\equiv a_{\chi _{20}}(T)$ mod $19$ for all $T$ with ${\rm tr}(T)\le 4$. \end{prop} Now we can show the following congruences.
\begin{prop} \label{Thm2} (1) $F_2^{(11)}\equiv -\chi _{12}$ mod $11$, (2) $F^{(19)}_2\equiv \chi _{20}$ mod $19$. \end{prop} \begin{proof} Since the proof is
similar, we show (1) only. It is known by B\"ocherer-Nagaoka \cite{Bo-Na} that there exists a modular form $G_{12}\in M_{12}(\Gamma _2)_{\mathbb{Z}_{(p)}}$ such that
$F^{(11)}_2\equiv G_{12} \pmod{11}$. From Proposition \ref{Na-Na}, $a_{F_{12}}(T)\equiv a_{-\chi_{12}}(T) \pmod{11}$ for all $T$ with ${\rm tr}(T)\le 5$.
Applying Theorem \ref{main1} to $F_{12}$ and $-\chi _{12}$, we have $G_{12}\equiv -\chi _{12} \pmod{11}$. In fact, $\frac{k}{10}=\frac{12}{10}$. Hence it
suffices to prove that we check the congruence of the Fourier coefficients for all $T$ with ${\rm tr}(T)\le 2$. Thus we obtain $F_2^{(11)}\equiv
G_{12}\equiv -\chi _{12} \pmod{11}$. This completes the proof of (1).  \end{proof}

\begin{rem} In \cite{Bo-Na}, B\"ocherer-Nagaoka expected that for each $F_2\in M_2(\Gamma ^{(2)}_{0}(p))_{\mathbb{Z}_{(p)}}$, there exists a
$G_{p+1}\in M_{p+1}(\Gamma _2)_{\mathbb{Z}_{(p)}}$ such that $F_2\equiv G_{p+1} \pmod{p}$ and studied it. \end{rem}

 \end{document}